\documentclass[12pt,reqno]{amsart}
\usepackage{amsmath,amsthm,amssymb,amscd,enumerate}
\usepackage{graphicx}

\newtheorem{theorem}{Theorem}

\newtheorem{lemma}{Lemma}
\newtheorem{corollary}{Corollary}

\newtheorem{remark}{Remark}

\newtheorem*{theorem*}{Theorem}

\newcommand{\R}{\mathbb R}
\newcommand{\e}{\varepsilon}
\newcommand{\vp}{\varphi}
\newcommand{\n}{\eta}

\title[Principal values of Riesz transforms]{Non existence of principal values of signed Riesz transforms of non integer dimension}

\author[A.Ruiz de Villa \and X.Tolsa]{Aleix Ruiz de Villa and Xavier Tolsa}

\address{Aleix Ruiz de Villa, Departament de Matem\`atiques, Universitat Aut\`onoma de
Bar\-ce\-lo\-na, 08193 Bellaterra (Barcelona), Catalunya}

\email{aleixrv@mat.uab.cat}

\address{Xavier Tolsa,
Instituci\'o Catalana de Recerca i Estudis Avan\c{c}ats (ICREA) and Departament de Matem\`atiques, Universitat Aut\`onoma de
Bar\-ce\-lo\-na, 08193 Bellaterra (Barcelona), Catalunya}

\email{xtolsa@mat.uab.cat}

\thanks{A. Ruiz de Villa was supported by grant AP-2004-5141. Also, both authors were partially supported by grants MTM2007-62817 (Spain) and 2005-SGR-00774 (Gene\-ra\-litat
de Catalunya)}

\begin{document}
\maketitle

\begin{abstract}
In this paper we prove that, given $s\geq 0$, and a Borel non zero measure $\mu$ in $\R^m$, if for $\mu$-almost every $x \in \R^m$ the limit $$\lim_{\e \to 0} \int_{|x-y|>\e}\frac{x-y}{|x-y|^{s+1}}d\mu(y)$$ exists and $0<\limsup_{r \to 0}\frac{\mu(B(x,r))}{r^s}<\infty$, then $s$ in an integer. In particular, if $E \subset \R^m$ is a set with positive and bounded $s$-dimensional Hausdorff measure $H^s$ and for $H^s$-almost every $x \in E$ the limit $$\lim_{\e \to 0} \int_{|x-y|>\e}\frac{x-y}{|x-y|^{s+1}}d H^s_{|E}(y)$$ exists, then $s$ is integer.
\end{abstract}

\begin{section}{Introduction}

Given a Borel measure $\mu$ in $\R^m$ and  $0<s\leq m$,
the $s$-Riesz transform of $\mu$ is
$$R^s \mu (x)=\int \frac{x-y}{|x-y|^{s+1}}\,d\mu(y),\qquad
x \notin \textrm{supp}(\mu) .$$
 Since for $x$ in the support of $\mu$ the integral may not be convergent, for $\e>0$ one considers the truncated Riesz transform
$$R_{\e}^s\mu (x)=\int_{|x-y|>\e} \frac{x-y}{|x-y|^{s+1}}\,d\mu(y),\qquad x\in\R^m.$$

The lower and upper $s$-dimensional densities of $\mu$ at $x$ are defined by
$$\theta_{\mu,*}^s(x)=\liminf_{r \to 0} \frac{\mu(B(x,r))}{r^s},\qquad \theta_{\mu}^{s,*}(x)=\limsup_{r \to 0} \frac{\mu(B(x,r))}{r^s}.$$ In the case where $\theta_{\mu,*}^s(x)=\theta_{\mu}^{s,*}(x)$ one calls this quantity the ($s$-dimensional) density of the measure $\mu$ at $x$, denoted by $\theta^s_\mu(x)$.

 The main result of this paper is the following.

\begin{theorem}\label{main3}
For $0\leq s \leq m$, 
let $\mu$ be a finite Radon measure in $\R^m$ such that $0<\theta^{s,*}_\mu (x)<\infty$ and $\lim_{\e \rightarrow 0} R^s_\e \mu(x) \mbox{  exists}$ for all $x$ in a set of positive $\mu$-measure. Then $s \in \mathbb Z$.
\end{theorem}

Consider now the case the case where $\mu$ coincides with the $s$-dimensional Hausdorff measure $H^s$ on a set $E$ with $0<H^s(E)<\infty$. Recall that for $H^s$-almost every $x \in E$ we have $0<\theta_{H^s_{|E}}^{s,*}(x)<\infty$.
So we deduce the following corollary.

\begin{corollary} \label{coro}
For $0\leq s \leq m$, let $E \subset \R^m$ be a set satisfying $0<H^s(E)<\infty$ such that for $\mathcal H^s$-almost every $x \in E$ the limit $$\lim_{\e \to 0} \int_{|x-y|>\e}\frac{x-y}{|x-y|^{s+1}}dH^s_{|E}(y)$$ exists.  Then $s \in \mathbb Z$.
\end{corollary}

Let us remark that Mattila and Preiss \cite{MP} already proved that if one assumes
\begin{equation}\label{infden} 
\theta_{\mu,*}^s(x) > 0 \qquad \mu{\rm -a.e. }\;\; x\in \R^n
\end{equation}
(instead of $\theta_{\mu}^{s,*}(x) > 0$ $\mu$-a.e.), then the $\mu$-a.e. existence of the principal value $\lim_{\e \rightarrow 0} R^s_\e \mu(x)$ forces
$s$ to be an integer. Later on, Vihtil\"a \cite{V} showed that this also holds if one assumes \eqref{infden} and
\begin{equation}\label{hip2}
\sup_{\e>0} |R_{\e}^s\mu (x)|<\infty \qquad \mu{\rm -a.e. }\;\; x\in \R^n
\end{equation}
(instead of the existence of the principal value $\lim_{\e \rightarrow 0} R^s_\e \mu(x)$ $\mu$-a.e.).
The proofs in \cite{MP} and \cite{V} rely on the use of tangent measures, and for these arguments, and for all usual arguments involving tangent measures, the assumption \eqref{infden} on the lower density is essential. So to prove theorem \ref{main3} we have followed 
a quite different approach, inspired in part by some of the techniques used in \cite{T1} and \cite{T2}. However, we have not been able to use the weaker assumption \eqref{hip2} instead of the one concerning the existence of principal values.

On the other hand, the case $0\leq s\leq 1$ of theorem \ref{main3}
follows from Prat's results \cite{P1}, \cite{P2}. In this case, the so called curvature method works, and one 
can even assume \eqref{hip2} instead of the fact that principal value
$\lim_{\e \rightarrow 0} R^s_\e \mu(x)$ exists $\mu$-a.e.

If one combines corollary \ref{coro} with the results in \cite{MM} and
\cite{T1} one gets:

\begin{theorem*}\label{junt}
For $0< s \leq m$, let $E \subset \R^m$ be a set satisfying $0<H^s(E)<\infty$.
The principal value
$$\lim_{\e \to 0} \int_{|x-y|>\e}\frac{x-y}{|x-y|^{s+1}}dH^s_{|E}(y)$$ exists
for $H^s$-almost every $x \in E$ if and only if $s$ is integer and $E$ is $s$-rectifiable.
\end{theorem*}

Recall that $E\subset\R^m$ is called  $s$-rectifiable if it is contained $H^s$-a.e. in a countable union of $s$-dimensional $C^1$-submanifolds of
$\R^m$. See also \cite{MP} for other previous results concerning the case $s$ integer, and \cite{M}, \cite{T0}, for the case $s=1$.

It is interesting to compare the last theorem with well known results in geometric measure theory due essentially to Marstrand \cite{Marstrand} and Preiss \cite{Preiss}:

\noindent{\em For $0< s \leq m$, let $E \subset \R^m$ be a set satisfying $0<H^s(E)<\infty$.
The density $\theta^s_{H^s|E}(x)$ exists 
for $H^s$-almost every $x \in E$ if and only if $s$ is integer and $E$ is $s$-rectifiable.}

\noindent Notice the analogies between this statement and the previous theorem.
\end{section}

\begin{section}{Main tools}

Given two different quantities $a,b$ we use the notation $a\lesssim b$ if there exists a fixed constant $C>0$ satisfying $a\leq C b$, with $C$ depending at most on $m$ and $s$. If also $b\lesssim a$, then we write $a\approx b$. Given $x \in \R^m$ and $r>0$, $B(x,r)$ stands for the open ball of center $x$ and radius $r$, and $\theta^s(x,r):=\mu(B(x,r))/r^s$ stands for the (average) $s$-dimensional density of the ball $B(x,r)$. In the case $x=0$ we write $\theta^s(r)=\theta^s(0,r)$. Throughout the paper $n$ will denote an integer satisfying $n< s \leq n+1\leq m$.

Given $0<\rho<1/2$ small enough, which will be fixed below, consider a function $\vp \in \mathcal C^2(0,\infty)$ satisfying:
\begin{itemize}
\item [(i)] $\vp(r)=r^{(s+1)/2}$ if $0\leq r\leq 1$,
\item [(ii)] $\vp(r)=-\frac{r}{\rho}+1+\rho+\frac{1}{\rho}$ if $ 1+\rho^2\leq r\leq 1+\rho^2+\rho$,
\item [(iii)] supp$(\vp)\subset [0, 1+ \rho+ 2\rho^2]$, $|\vp(r)|\leq C$, $|\vp'(r)|\leq 1 / \rho$ and $|\vp''(r)|\leq C_\rho$  for all $r>0$, where $C_\rho$ depends on $\rho$.
\end{itemize}
See fig. \ref{fig}.
\begin{figure}
\begin{center}
\includegraphics[width=9.5cm]{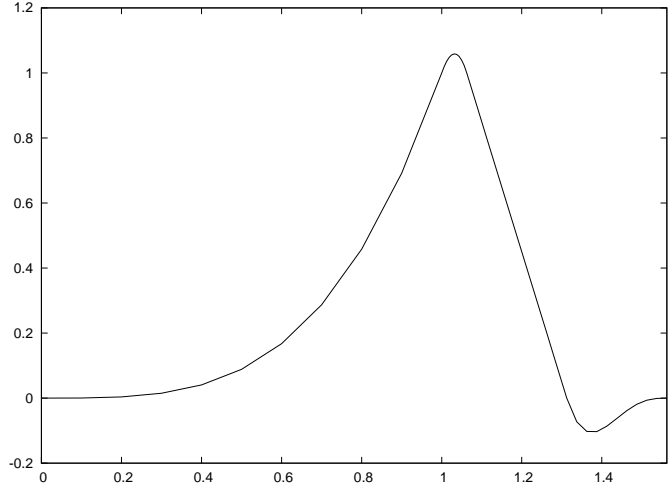}
\end{center}
\caption{\label{fig} Example of the function $\vp$ for the values $\rho=1/4$ and $s=6$.}
\end{figure}

Given $\e>0$, consider the operator:
$$R_{\vp,\e}^s\mu(x)=\int \vp\left( \frac{|x-y|^2}{\e^2}\right)\frac{x-y}{|x-y|^{s+1}}d\mu(y)=\int k_{\vp,\e}(x-y)d\mu(y).$$
Notice that $k_{\vp,\e}$ is a kernel supported on $B(0,3\e)$ satisfying $\|k_{\vp,\e}\|_\infty\leq C / \e^s$ and
\begin{equation}\label{kernel}
\|\nabla k_{\vp,\e}\|_\infty\leq C(\rho) / \e^{s+1}.
\end{equation}
Also observe that
\begin{align*}
R_{\vp,\e}^s\mu(x)&=\int \int_{0< t <\frac{|x-y|^2}{\e^2}}\vp'(t)dt \frac{x-y}{|x-y|^{s+1}}d\mu(y)=\int \vp'(t) R_{\e\sqrt{t}}^s \mu(x)dt.
\end{align*}
Using the fact that $\int \vp'(t)dt < \infty$ and $ \sup_{\e>0}|R^s_\e \mu(x)|<\infty$ for $\mu$-almost all $x \in\R^m$, we conclude that if $\lim_{\e \rightarrow 0} R^s_\e \mu(x)$ exists, then $\lim_{\e \rightarrow 0} R^s_{\vp,\e} \mu(x)$ also exists.

Given $C_0, r_0, \e_0>0$  and $0<\delta<1$, set
\begin{align}\label{defdelta}
F_\delta:=&\bigl\{x \in \R^m: \mu(B(x,r))\leq 2\theta^{s,*}_\mu(x) r^s \textrm{ for all } r \leq r_0, \nonumber  \\ &|R_{\vp,\e}^s\mu(x)-R_{\vp,\e'}^s\mu(x)|\leq \delta \textrm{ for all } \e,\e' \leq \e_0, \textrm{ and } \theta_\mu^{s,*}(x) \leq C_0\bigr\}.
\end{align}
If $r_0$ and $\e_0$ are small enough and $C_0$ is big enough, we have $\mu(F_\delta) >0$. Also observe that if $x \in F_\delta$, for all $r>0$,
\begin{equation}\label{M}
\mu(B(x,r)) \leq M r^s,
\end{equation}
where $M=\max\{2C_0, \mu(\R^m)/r_0^s\}$.

\begin{lemma}\label{taylorlemma}
Suppose that $0 \in F_\delta$ and let $x \in B(0,\e/4)$, then:
\begin{equation}\label{taylor}
R_{\vp,\e}^s\mu(x)-R_{\vp,\e}^s\mu(0)=T^\e(x)+E(x),
\end{equation}
where
\begin{equation}
T^\e(x)=\int \frac{1}{|y|^{s+1}}\left[ \vp\left( \frac{|y|^2}{\e^2}\right)\left(x-\frac{(s+1)(x\cdot y)y}{|y|^2} \right)+\vp'\left( \frac{|y|^2}{\e^2}\right)\frac{2(x\cdot y)y}{\e^2} \right]d\mu(y),
\end{equation}
and
$$|E(x)|\leq C_1 \theta^s(3 \e) \frac{|x|^2}{\e^2}.$$
The constant $C_1$ only depends on $\rho$ (and also $C_\rho$).
\end{lemma}

\begin{proof}
We will prove equality (\ref{taylor}) as in \cite{T2}. Applying Taylor's formula to the function $g(t)=\vp(t)/t^{(s+1)/2}$ at a point $t_0>0$ we have $$\frac{\vp(t)}{t^{(s+1)/2}}=\frac{\vp(t_0)}{t_0^{(s+1)/2}}+\frac{t_0\vp'(t_0)-(s+1)\vp(t_0)/2}{t_0^{(s+3)/2}}(t-t_0)+g''(\xi)\frac{(t-t_0)^2}{2},$$
for some $\xi \in [t,t_0]$. Notice that if $0<t\leq 1$, then $\vp(t)/t^{(s+1)/2}=1$. Setting $t=|x-y|^2/\e^2$ and $t_0=|y|^2/2$, and multiplying by the vector $(x-y)$ we get
\begin{align*}
\vp\left(\frac{|x-y|^2}{\e^2}\right)\frac{x-y}{|x-y|^{s+1}}= & \,\vp \left( \frac{|y|^2}{\e^2}\right)\frac{-y}{|y|^{s+1}}+\vp \left( \frac{|y|^2}{\e^2}\right)\frac{x}{|y|^{s+1}}\\ &+\frac{|y|^2\vp'\left( \frac{|y|^2}{\e^2}\right)/\e^2-(s+1)\vp\left( \frac{|y|^2}{\e^2}\right)/2}{|y|^{s+3}}(|x|^2-2x\cdot y)(x-y)\\ &+ g''(\xi_{x,y})\frac{(|x|^2-2x\cdot y)^2}{2 \e^{s+5}}(x-y),
\end{align*}
where $\xi_{x,y} \in [|x-y|^2/\e^2,|y|^2/\e^2]$. Integrating with respect to $y$ we obtain (\ref{taylor}) with $$E(x)=\int E(x,y) d\mu(y),$$ where
\begin{align*}
E(x,y)&=\frac{1}{|y|^{s+3}}\left[\frac{|y|^2}{\e^2}\vp'\left(\frac{|y|^2}{\e^2}\right)-\frac{s+1}{2}\vp\left(\frac{|y|^2}{\e^2} \right) \right][|x|^2x-|x|^2y-2(x\cdot y)x]\\& +g''(\xi_{x,y})\frac{(|x|^2-2 x \cdot y)^2}{2 \e^{s+5}}(x-y)=E_1(x,y)+E_2(x,y).
\end{align*}
For $i=1,2$ consider the decomposition
$$\int E_i(x,y) d\mu(y)= \left(\int_{|y|<  \e/2}+\int_{\e/2 \leq |y|\leq 3\e}\right) E_i(x,y) d\mu(y)=A_i+B_i.$$
Let us estimate $E_1$ first.

\begin{itemize}
\item [(a)] If $|y|\leq\e/2$, using that $\vp(r)=r^{(s+1)/2}$ for $0<r\leq 1$, $$\frac{|y|^2}{\e^2}\vp'\left(\frac{|y|^2}{\e^2}\right)-\frac{s+1}{2}\vp\left(\frac{|y|^2}{\e^2} \right)=0,$$
    thus $A_1=0$.
\item [(b)] If $|y|>\e/2$, we have $$\left|\frac{|y|^2}{\e^2}\vp'\left(\frac{|y|^2}{\e^2}\right)-\frac{s+1}{2}\vp\left(\frac{|y|^2}{\e^2} \right) \right| \leq C=C(\rho).$$ Since $|x|<\e/4$, then $||x|^2x-|x|^2y-2(x\cdot y)x|\leq C |y| |x|^2 $. Moreover, recall that supp$(\vp) \subset [0,3]$. As a consequence,
    $$|B_1|\leq C|x|^2 \int_{\e/2 \leq |y|\leq 3\e} \frac{1}{|y|^{s+2}}d \mu(y)\leq C \theta^s(3 \e) \frac{|x|^2}{\e^2}.$$
\end{itemize}
We now estimate $E_2$. Recall that $$E_2(x,y)=g''(\xi_{x,y})\frac{(|x|^2-2 x \cdot y)^2}{2 \e^4}(x-y),$$ with $\xi_{x,y} \in [|y|^2/\e^2, |x-y|^2/\e^2]$ and $$g''(r)=\frac{r^2 \vp''(r)-(s+1)r \vp'(r)+\frac{(s+1)(s+3)}{4}\vp(r)}{r^{(s+5)/2}}.$$ Denote $t=\max\{|y|,|x-y|\}$.
\begin{itemize}
\item [(a)] If $|y|\leq \e/2$, we have $|\xi_{x,y}|<1$ and thus $|g''(\xi_{x,y})| = 0$. So $A_2 =0$.

\item [(b)] If $|y|>\e/2$, we have $\xi_{x,y} \approx \frac{|y|^2}{\e^2}$, and so $|g''(\xi_{x,y})|\leq C(\rho) (\e/|y|)^{s+5}$. Moreover, if $|y|>3\e$, then $|x-y|>2\e$ and so $\xi_{x,y}>4$, which implies that $g''(\xi_{x,y}) = 0$. On the other hand,
    $$|(|x|^2-2 x \cdot y)^2(x-y)|\leq C |(|x|^2+|x||y|)^2(|x|+|y|)|\leq C |x|^2|y|^3.$$
    Therefore,
    \begin{align*}
    |B_2|&\leq C  \frac{|x|^2}{\e^{s+5}} \int_{\e/2 \leq |y|\leq 3\e} \left(\frac{\e}{|y|} \right)^{s+5}|y|^3 d\mu(y)\\ &\leq C |x|^2 \int_{\e/2 \leq |y|\leq 3\e} \frac{1}{|y|^{s+2}}d\mu(y)\leq C \theta^s(3 \e)\frac{|x|^2}{\e^2}.
    \end{align*}
\end{itemize}
\end{proof}

To prove theorem \ref{main3}, we will find a ball with high average density and an $n$-dimensional hyperplane $L$ such that all the points in the ball are close to $L$. Estimating densities from above and below, we will get a contradiction. We need the following auxiliary result.

\begin{lemma}\label{punts}
Suppose that $\mu(B(x_0,r)\cap F_\delta)\geq C_2 r^s$ and $n<s\leq n+1\leq m$. Then there exist a constant $C_3>0$ depending on $n,s,C_2$ and $M$ (from the equation (\ref{M})), and $n+2$ points $y_0,\ldots, y_{n+1} \in B(x_0,r) \cap F_\delta$ such that for $j=1,\ldots,n+1$
\begin{equation}\label{distancia}d(y_j,L_{j-1})\geq C_3 r, \end{equation}  where $L_j$ stands for the $j$-dimensional hyperplane that contains $y_0,\ldots,y_j$
\end{lemma}

\begin{proof}
The proof of this lemma can be found in \cite{DS} (chapter 5, p. 28). For completeness we recall the arguments. We will use induction. Take $1\leq j \leq n$ and suppose that there exist $y_0, \ldots, y_j \in B(x_0,r)\cap F_\delta$ satisfying (\ref{distancia}) and such that for all $y \in B(x_0,r)\cap F_\delta$, denoting $L_j=\langle y_0, \ldots, y_j\rangle$, $$d(y,L_j)< \nu r$$ with $\nu>0$ to be chosen below. Then $B(x_0,r)\cap F_\delta$ can be covered by $C/ \nu^j$ balls with radius $\nu r$, so using the polynomial growth of degree $s$ of the measure,
$$C_2r^s \leq \mu(B(x_0,r)\cap F_\delta) \leq \frac{CM}{\nu^j} (\nu r)^s.$$
Taking $\nu < C(C_2/M)^{1/(s-j)}$ we get a contradiction.
\end{proof}

Below we will use the following notation. Given points $y_0,\ldots, y_k$, the $k$-dimensional hyperplane which contains these points is  
$\langle y_0,\ldots,y_k\rangle$. On the other hand, given vectors $u_1,\ldots,u_k$, the subspace spanned by $u_1,\ldots,u_k$ is denoted by $[u_1,\ldots,u_k]$. So we have
$\langle y_0,\ldots,y_k\rangle = y_0 + [y_1-y_0,\ldots,y_k-y_0]$.


\begin{lemma}\label{acotpral}
Suppose that $\mu(B(x_0,r)\cap F_\delta)\geq C_2 r^s$ and $r\leq \e/20$.  Consider points $y_0, \ldots, y_{n+1} \in B(x_0,r)\cap F_\delta$ and hyperplanes $L_0,\ldots, L_{n+1}$ satisfying (\ref{distancia}), like in lemma \ref{punts} (in particular $L_n=\langle y_0, \ldots, y_n \rangle$ and $L_{n+1}=\langle y_0, \ldots, y_{n+1} \rangle$). Then we have
\begin{equation}\label{acotpraleq}
\textrm{d}(y_{n+1},L_n) |U^\e(y_0)|\leq C_4 \left(\sum_{j=1}^{n+1}|R^s_{\vp,\e}(y_j)-R^s_{\vp,\e}(y_0)| +\theta^s(y_0,3\e)\frac{r^2}{\e^2}\right),
\end{equation}
where $C_4$ depends on $C_2$ and $M$, and denoting by $\Pi_{L_{n+1}}(z) $ the orthogonal projection of $z$ onto $L_{n+1}$,
\begin{align*}
U^\e(y_0)&=\int  \frac{1}{|z-y_0|^{s+1}}\biggl[ \vp\left( \frac{|z-y_0|^2}{\e^2} \right)\left( (n+1)-(s+1)\frac{| \Pi_{L_{n+1}}(z-y_0) |^2 }{|z-y_0|^2}\right)\\ &+ 2 \vp'\left( \frac{|z-y_0|^2}{\e^2} \right)\frac{| \Pi_{L_{n+1}}(z-y_0) |^2 }{\e^2} \biggr] d\mu(z).
\end{align*}
\end{lemma}

\begin{proof}
 Suppose without loss of generality that $y_0=0$. Consider orthonormal vectors $e_1,\ldots, e_{n+1}$ such that $L_k=[ e_1,\ldots,e_k]$ for $k=1,\ldots, n+1$. Moreover, take $e_{n+1}=(y_{n+1}-u)/|y_{n+1}-u|$, where $u$ denotes the orthogonal projection of $y_{n+1}$ onto $L_n$.

Observe that, denoting $z_{(i)}=z\cdot e_i$ for $i=1,\ldots, n+1$,
\begin{align*}
U^\e(0)&=\int  \frac{1}{|z|^{s+1}}\biggl[ \vp\left( \frac{|z|^2}{\e^2} \right)\left( (n+1)-(s+1)\frac{\sum_{k=1}^{n+1}z_{(k)}^2}{|z|^2}\right)\\ &+ 2 \vp'\left( \frac{|z|^2}{\e^2} \right)\frac{\sum_{k=1}^{n+1}z_{(k)}^2}{\e^2} \biggr] d\mu(z)=\sum_{k=1}^{n+1}T^\e(e_k)\cdot e_k.
\end{align*}

To show (\ref{acotpraleq}), we will estimate $U^\e(y_0)$ from above using lemma \ref{taylorlemma}. Let us prove it by induction on $k$ ($k \leq n$):
\begin{equation}\label{modorto}
|T^\e(e_k)\cdot e_k| \leq |T^\e(e_k)| \lesssim \frac{1}{r} \sum_{j=1}^{k} |T^\e(y_j)|.
\end{equation}
For $k=1$ we write $e_1=y_1/|y_1|$. Since $|y_1|=\textrm{dist}(y_1,0) \geq C r$, $$|T^\e(e_1)| \lesssim \frac{1}{r}|T^\e(y_1)|.$$
Now suppose that equation (\ref{modorto}) holds for $k-1$. There exist $\lambda_j, \tilde \lambda_j \in \R$, with $\lambda_k \neq 0$, such that
$$e_k=\lambda_k y_k +\sum_{j=1}^{k-1}\lambda_j y_j=\lambda_k y_k +\sum_{j=1}^{k-1}\tilde \lambda_j e_j,$$
and so $$y_k=\frac{1}{\lambda_k} e_k - \sum_{j=1}^{k-1}\frac{\tilde \lambda_j}{\lambda_k} e_j.$$
Then,
$$\frac{1}{|\lambda_k|}=| y_k\cdot e_k|=\textrm{dist}(y_k, L_{k-1}) \geq  C r,$$
so $$|\lambda_k| \lesssim 1/r.$$
On the other hand, for $j=1,\ldots,k-1$, $$0= e_k\cdot e_j=\lambda_k  y_k\cdot e_j+\tilde \lambda_j,$$ and so $$|\tilde \lambda_j|=|\lambda_k y_k\cdot e_j |\leq C.$$
Finally,
\begin{align*}
|T^\e(e_k)|\lesssim \frac{1}{r}|T^\e(y_k)|+\bigl|\sum_{j=1}^{k-1} T^\e(e_j)\bigr|\lesssim \frac{1}{r} \sum_j |T^\e(y_j)|.
\end{align*}
Now, since $u \in L_n=[ e_1,\ldots,e_n ]$ there exist $\overline \lambda_1,\ldots,\overline \lambda_n$ with $|\overline \lambda_i|\leq C r$ for $i=1, \ldots, n$ such that $u=\sum_{i=1}^n \overline \lambda_i e_i$. Therefore,
\begin{align*}
|T^\e(e_{n+1})|&=\frac{1}{\textrm{dist}(y_{n+1},L_n)}|T(y_{n+1})-T(u)|\\& \lesssim \frac{1}{\textrm{dist}(y_{n+1},L_n)}\left(\sum_{j=1}^{n} |T^\e(y_j)|+|T^\e(y_{n+1})|\right).
\end{align*}
Applying lemma \ref{taylorlemma}, since $|y_j|\leq r$ for $i=1,\ldots,n+1$, we finally have
\begin{align*}
|U^\e(0)|=|\sum_{k=1}^{n+1}T^\e(e_k) e_k |\lesssim \frac{1}{\textrm{dist}(y_{n+1},L_n)}\left(\sum_{j=1}^{n+1}|R^s_{\vp,\e}(y_j)-R^s_{\vp,\e}(0)|+\theta^s(3\e)\frac{r^2}{\e^2}\right).
\end{align*}
\end{proof}

The following key lemma gives us a estimate from below of the term $|U^\e(y_0)|$.

\begin{lemma}\label{1beta}
Suppose that $\mu(B(x_0,r)) \geq C_2 r^s$ and consider points $y_0, \ldots, y_{n+1} \in B(x_0,r)\cap F_\delta$ as in Lemma \ref{acotpral}, and let
$\e_1=r/\tau$ with $\tau<1/4$. If $\rho>0$ is a constant small enough (depending only on $s$), then there exists an $\omega_0=\omega_0(\tau,s,\rho, M, C_2)\geq 1$ such that we can find an $\e>0$ satisfying $\e_1 \leq \e \leq \omega_0 \e_1$, $\theta^s(y_0,4 \e) \leq C \theta^s(y_0,\e)$, $\theta^s(y_0,\e) \geq C_2 \tau^s /2 $, and
$$|U^{\e}(y_0)|\geq\frac{7}{10} \frac{ \theta^s(y_0,\e)(n+1-s)} {\e}.$$
\end{lemma}

\begin{remark}
Notice that this lemma is useful only when $s$ is non integer, that is when $n<s<n+1$. This is one of the key steps of the proof of theorem \ref{main3}, where there are differences between the integer and the non integer case.
\end{remark}

\begin{proof}[Proof of lemma \ref{1beta}]
Clearly we may assume $s \neq n+1$. Also, we suppose that $y_0=0$. For $k \geq 0$, let us denote
$$\delta_k=\sup_{\e_1 \leq t \leq 4^k \e_1} \frac{\mu(B(0, t))}{t^s}.$$
Suppose that for all $k\geq 0$ we have $\delta_k \leq \delta_{k+1} / (1+\rho^2/4)$. Then, since $\delta_0 \geq C_2 \tau^s$,
\begin{align*}
C_2 \tau^s (1+\rho^2/4)^k \leq \delta_k \leq M,
\end{align*}
which leads to contradiction for $k$ big enough. Thus, there exists $\omega_0=\omega_0(\tau,s,\rho, M, C_2)>0$ and there exists $1\leq k\leq \log_4 \omega_0$ such that $\delta_k \geq \delta_{k+1} / (1+\rho^2/4)$. Take $\e \in [\e_1, 4^k \e_1]$ such that $ \delta_k \leq \theta^s(\e) (1+\rho^2/4)$. Then, $\theta^s(\e)\geq  \tau^s\mu(B(x_0,r))/(2r^s) \geq  C_2\tau^s/2$, and also for all $t$ such that $\e\leq t \leq 4\e$ we have
\begin{equation}\label{1betaeq}
\theta^s(t)=\frac{\mu(B(0,t))}{t^s}\leq \delta_{k+1}\leq \delta_k (1+\rho^2/4) \leq \theta^s(\e)(1+\rho^2).
\end{equation}
Given orthonormal vectors $\{e_i\}_{i=1}^{n+1}$ such that $[e_1,\ldots, e_{n+1}]=[y_1-y_0,\ldots, y_{n+1}-y_0]$, we denote
$$g_\e(z)=\frac{1}{|z|^{s+1}}\left[ \vp\left( \frac{|z|^2}{\e^2} \right)\left( (n+1)-(s+1)\frac{\sum_{i=1}^{n+1}z_{(i)}^2}{|z|^2}\right)+2 \vp'\left( \frac{|z|^2}{\e^2} \right)\frac{\sum_{i=1}^{n+1}z_{(i)}^2}{\e^2} \right],$$
where $z_{(i)}=z\cdot e_i$.
Consider the following domains:
\begin{itemize}
\item $A_1:=\{z \in \R^m: |z|< \e\}$,
\item $A_2:=\{z \in \R^m: \e\leq|z|\leq \e \sqrt{1+\rho^2}\}$,
\item $A_3:=\{z \in \R^m: \e\sqrt{1+\rho^2}<|z|<\e \sqrt{1+\rho+\rho^2}\}$,
\item $A_4:=\{z \in \R^m: \e\sqrt{1+\rho+\rho^2}<|z|<\e \sqrt{1+\rho+2\rho^2}\}$.
\end{itemize}
Then,
\begin{align*}
U^\e(0)&= \sum_{i=1}^4 \int_{A_i} g_\e(z)d\mu(z)=:I_1+I_2+I_3+I_4\geq I_1-|I_2|-|I_3|-|I_4|.
\end{align*}

First we consider $I_1$:
\begin{align*}
I_1&=\frac{1}{\e^{s+1}}\int_{|z|<\e} \frac{1}{|z|^{s+1}} \left[ |z|^{s+1}\left( (n+1)-(s+1)\frac{\sum_{j=1}^{n+1}z_{(j)}^2}{|z|^2} \right)+(s+1)|z|^{s-1} \sum_{j=1}^{n+1}z_{(j)}^2\right]d\mu(z)\\ &=\frac{n+1}{\e^{s+1}}\int_{|z|<\e}d\mu(z)=\frac{(n+1)\theta^s(\e)}{\e}.
\end{align*}

Now we estimate $I_2$ using the fact that for all $r>0$, $| \vp(r)| \leq C$ and $|\vp'(r)|\leq 1 / \rho$:
$$|I_2|\leq C \frac{1}{\rho \e^{s+1}}\mu (B(0,\e(1+2\rho^2)) \backslash B(0,\e))\leq  C \frac{\theta^s(\e)}{ \e}\left(\frac{(1+\rho^2)(1+2\rho^2)^s}{\rho}-\frac{1}{\rho}\right).$$
So, if $\rho$ is small enough, $$|I_2|\leq \frac{(n+1-s)\theta^s(\e)}{10 \e}.$$

Let us deal with $I_3$. Recall that in $A_3$, $|\vp ( \frac{|z|^2}{\e^2} )|=|-\frac{|z|^2}{\e^2\rho}+1+\rho+\frac{1}{\rho}|\leq 1$ and $|\vp' ( \frac{|z|^2}{\e^2} )|=\frac{1}{\rho}$. Using that for $z \in A_3$, $\sum_{i=1}^{n+1}z_{(i)}^2\ \leq \e^2(1+\rho+\rho^2)$ and $|z|^2\geq \e^2(1+\rho^2)$, we obtain
\begin{align*}
&|I_3|\leq\int_{A_3} \frac{1}{|z|^{s+1}}\left| (n+1)-(s+1)\frac{\sum_{i=1}^{n+1}z_{(i)}^2}{|z|^2}\right|+\frac{2}{\rho|z|^{s+1}} \left|\frac{\sum_{i=1}^{n+1}z_{(i)}^2}{\e^2}\right|d\mu(z)\\ &\leq \frac{(n+s+2) \mu(A_3)}{\e^{s+1}(1+\rho^2)^{(s+1)/2}}+ \frac{1+\rho+\rho^2}{\e^{s+1}(1+\rho^2)^{(s+1)/2}}\frac{2\mu(A_3)}{\rho} :=I_3^1+I_3^2.
\end{align*}
Observe that, by (\ref{1betaeq}), we have
\begin{align*}
\frac{\mu(A_3)}{\e^{s+1}} & \leq \frac{1}{\e^{s+1}}\left( \mu(B(0,\e \sqrt{1+\rho+\rho^2})\backslash B(0,\e))\right)\\ &\leq \frac{\theta^s(\e)}{\e}\left( (1+\rho^2)(1+\rho+\rho^2)^{s/2}-1\right).
\end{align*}
So
$$|I_3^1|\leq \frac{(n+1-s)\theta^s(\e)}{10 \e},$$ provided by $\rho$ is small enough. On the other hand, by (\ref{1betaeq}) again,
\begin{align*}
\frac{2\mu(A_3)}{\e^s \rho}&\leq\frac{2}{\e^s\rho}\left(\mu(B(0,\e \sqrt{1+\rho+\rho^2})\backslash B(0,\e))\right)\\ &\leq \frac{2\theta^s(\e)}{ \rho}((1+\rho) ( \sqrt{1+\rho+\rho^2})^s-1)
\end{align*}
Since $\lim_{\rho \to 0} \frac{ 2(1+\rho+\rho^2)^{s/2}-1}{\rho}=s$, we deduce  $$|I_3^2|\leq \frac{s \theta^s(\e)}{\e}+\frac{(n+1-s)\theta^s(\e)}{10 \e},$$
for $\rho$ small enough.

Using similar arguments to the ones used to estimate $|I_2|$ and $|I_3|$ we deduce that $$|I_4|\leq \frac{(n+1-s)\theta^s(\e)}{10 \e},$$ for $\rho$ small enough.

We conclude that
$$U^\e(0) \geq \frac{\theta^s(\e)}{\e}\left(n+1-s-\frac{3}{10}(n+1-s) \right)=\frac{7}{10}\frac{(n+1-s)\theta^s(\e)}{\e},$$ so taking $\rho$ small enough we are done.
\end{proof}

\begin{remark}
In the proof of the preceding lemma the special form of the function $\varphi$ plays an important role. The choice of this function is one of the key points in our arguments. 
\end{remark}

In the following lemma we are strongly using the hypothesis that $\lim_{\e \rightarrow 0} R^s_{\vp,\e} \mu(x)$ exists $\mu$-a.e.

\begin{lemma}\label{cauchy}
Given $0<\delta<1/4$, $x_0 \in \R^m$ and $r>0$. If $\e , r/\delta<\e_0$, then for all $x,z \in B(x_0,r) \cap F_\delta$ we have
$$|R_{\vp,\e}^s\mu(x)-R_{\vp,\e}^s \mu(z)|\leq C_6 \delta,$$
with $C_6$ depending on $\rho$, $s$ and $M$.
\end{lemma}
\begin{proof}
Take $x,z \in B(x_0,r) \cap F_\delta$ and denote $\n= r/\delta$. By (\ref{kernel}) and using that $B(z, 3\n), B(x,3\n) \subset B(x_0,4 \n)$,
\begin{align*}
|R_{\vp,\n}^s\mu(x)-R_{\vp,\n}^s \mu(z)|&\leq \int |k_{\vp,\n}(x-y)-k_{\vp,\n}(z-y)|d\mu(y) \leq |z-x| \|\nabla k_{\vp,\n} \|_\infty \mu(B(x_0,4\n)) \\ &\leq\frac{C(\rho)M |z-x|}{\n}\leq C(\rho) M \delta.
\end{align*}
Now, since $x,z \in F_\delta$,
\begin{align*}
|R_{\vp,\e}^s\mu(x)-R_{\vp,\e}^s \mu(z)|&\leq |R_{\vp,\e}^s\mu(x)-R_{\vp,\n}^s \mu(x)|+|R_{\vp,\n}^s\mu(x)-R_{\vp,\n}^s \mu(z)|\\ &+|R_{\vp,\n}^s\mu(z)-R_{\vp,\e}^s \mu(z)|\leq C \delta.
\end{align*}

\end{proof}
\end{section}

\begin{section}{Proof of theorem 1}
Let $0<\delta, \tau<1/4$ to be chosen below and $\rho$ and $\omega_0=\omega_0(\tau, \rho)$ as in lemma \ref{1beta}. Consider the modified Riesz transform $R^s_{\vp,\e}$ depending on $\rho$. Suppose that $s$ is non integer, and so $n< s < n+1\leq m$. Let $x_0 \in F_\delta$ be a density point of $F_\delta$ with respect to $\mu$. Replacing $\mu$ by $\mu /\theta^{s,*}_\mu(x_0)$ if necessary, we may assume that $\theta^{s,*}_\mu(x_0)=1$. Take
\begin{equation}\label{1}
r< \e_0 \tau^{s+2}/ \omega_0 \mbox{  such that }\mu(B(x_0,r)\cap F_\delta) \geq r^s/2.
\end{equation}
Applying lemma \ref{punts} we can find $n+2$ points $y_0,\ldots,y_{n+1} \in B(x_0,r) \cap F_\delta$ such that
\begin{equation}\label{2}
\textrm{dist}(y_k,L_{k-1})\geq C r \textrm{ for } k=1,\ldots,n+1,
\end{equation}
where $L_k$ stands for the $k$-dimensional hyperplane that contains $y_0,\ldots,y_k$, and $C$ depends on $s,m,\theta^{s,*}_\mu(x_0)$ and the constant $M$ in \eqref{M} (which, in its turn, depends on the constants $r_0$ and $C_0$ in the definition of $F_\delta$ in \eqref{defdelta}, but not on $\varepsilon_0$!). Without loss of generality we suppose that $y_0=0$. Taking \begin{equation*}
\e_1=r/\tau,
\end{equation*}
by lemma \ref{1beta}, we can find $\e>0$ such that
\begin{equation}\label{4}
\e_1\leq \e\leq \omega_0 \e_1,
\end{equation}
\begin{equation}\label{5}
|U^\e(0)| \geq C\theta^s(\e) /\e,
\end{equation}
and
\begin{equation}\label{6}
\theta^s(\e) \geq C \tau^s \mbox{ and } \theta^s(4\e) \leq C \theta^s(\e).
\end{equation}
If we take
\begin{equation*}
\delta= \tau^{s+2}/ \omega_0,
\end{equation*}
(notice that $\omega_0$ is a large number, and so if $\tau$ is small enough, $\delta<1/4$), then we have
\begin{equation*}
\frac{r}{\delta}<\e_0\,\, \mbox{ by (\ref{1})},
\end{equation*}
and
\begin{equation*}
\e\leq \omega_0 \e_1 =\frac{\omega_0 r}{\tau}<\frac{\e_0 \tau^{s+2}}{\tau}<\e_0.
\end{equation*}
By (\ref{2}) and (\ref{5}), and lemmas \ref{acotpral} and \ref{cauchy}, we obtain
\begin{equation}\label{15}
\theta^s(\e)r\lesssim \e |U^\e(0)| \textrm{dist}(y_{n+1},L_n) \lesssim \e \delta+\theta^s(3\e)\frac{r^2}{\e}.
\end{equation}
By the definition of $\delta$ and $\e_1$, and by (\ref{4}) and (\ref{6}), we get
\begin{equation*}
\e \delta< \omega_0 \e_1 \frac{\tau^{s+2}}{\omega_0}=\e_1 \tau^{s+2}< r \theta^s(\e) \tau,
\end{equation*}
and by (\ref{4}) and (\ref{6}), and the definition of $\e_1$
\begin{equation*}
\theta^s(3\e)\frac{r^2}{\e}\lesssim \theta^s(\e)\frac{r^2}{\e_1}=\theta^s(\e)\tau r.
\end{equation*}
Thus, by (\ref{15}),
\begin{equation}
\theta^s(\e)r \lesssim \tau \theta^s(\e)r.
\end{equation}
Finally, taking $\tau$ small enough we get a contradiction.

\end{section}

\end{document}